\documentclass{amsart}
\usepackage{setspace}
\usepackage{amssymb,amsmath}
\usepackage{amsthm}
\makeatletter 
\@addtoreset{equation}{section}
\makeatother  
\newtheorem{theorem}{Theorem}[section]
\newtheorem{lemma}[theorem]{Lemma}

\newtheorem{corollary}{Corollary}[theorem]

\def\Xint#1{\mathchoice
  {\XXint\displaystyle\textstyle{#1}}%
  {\XXint\textstyle\scriptstyle{#1}}%
  {\XXint\scriptstyle\scriptscriptstyle{#1}}%
  {\XXint\scriptscriptstyle\scriptscriptstyle{#1}}%
  \!\int}
\def\XXint#1#2#3{{\setbox0=\hbox{$#1{#2#3}{\int}$}
  \vcenter{\hbox{$#2#3$}}\kern-.5\wd0}}

\def\dashint{\Xint-}
\author{Chengyang Shao}
\address{MIT, Dept. of Math.\\
77 Massachusetts Avenue, Cambridge, MA 02139-4307.}
\title{Schauder Type Estimates for a Class of Hypoelliptic Operators}
\begin{document}
\maketitle
\begin{spacing}{1.2}
\section{Introduction}
\subsection{Schauder Estimates and Generalizations}
The Schauder estimate is a fundamental tool for the regularity theory of elliptic and parabolic partial differential operators. Its classical and most widely used version gives a $C^{2,\alpha}$ \emph{a priori} bound
$$|u|_{2,\alpha;\Omega'}\leq C(a,b,c,\Omega',\Omega)(|u|_{0;\Omega}+|f|_{\alpha;\Omega})$$
for any $C^{2,\alpha}$ solution $u$ of a uniformly elliptic equation:
$$\sum_{i,j}a_{ij}D_iD_j u+\sum_{k}b_kD_ku+cu=f.$$
The constant $C$ depends on the ellipticity constants and $C^{\alpha}$-norms of the coefficients. Such estimates can be generalized in several ways, both to more general equations (weak form or parabolic equations) and to equations allowing weaker regularity assumption on coefficients.

The present paper aims to combine the results given by Wang in \cite{Wang} and Dong, Kim in \cite{DK} to give a generalization of the Schauder estimate for the higher-order hypo-elliptic operators treated by Simon in \cite{Simon1}, which include the so-called semi-elliptic operators (see H\"{o}rmander \cite{Hor}, pp. 67-68) as a special case. The results in the present paper apply to parabolic equations of higher order and, for example, operators like $\partial_t^2-\partial_x^3$. Generalization to systems is also directly obtained.

Let us briefly review the results cited above. The simplest proof for second-order elliptic and parabolic operators (with real coefficients) was given by Wang \cite{Wang} with an essential use of the maximum principle, allowing Dini-continuous coefficients and right-hand-side. He obtained a sharp estimate as follows: \emph{for any solution $u\in C^{2,1}(B_R)$ (parabolic quasi-ball) of $\partial_t u-\sum_{i,j} a_{ij}D_iD_ju=f$ with $a_{ij},f\in C(B_R)$,
$$
\begin{aligned}
|D^2 u(x,t)-D^2u(y,s)|&\leq C({R,R'})|u|_0\left(\int_0^r\frac{\omega_a(\rho)}{\rho}d\rho+r\int_0^R\frac{\omega_a(\rho)}{\rho^2}d\rho\right)\\
&+C({R,R'})\left(\int_0^r\frac{\omega_f(\rho)}{\rho}d\rho+r\int_0^R\frac{\omega_f(\rho)}{\rho^2}d\rho\right),
\end{aligned}
$$
where $\omega_a,\omega_f$ is the oscillation of $a_{ij},f$ on $B_R$ respectively, $x,y\in B_{R'}$ with $R'<R$ and $r$ is the parabolic distance between $(x,t)$ and $(y,s)$.}

Another proof by Simon \cite{Simon1}, though only concerned with the H\"{o}lder estimate, generalized the Schauder estimate to a class of higher-order hypo-elliptic operators with a neat blow-up argument. The result of Dong and Kim \cite{DK} was similar to Wang's, was obtained by a delicate modification of Campanato's method, and required weaker regularity assumptions on the coefficients and right-hand-side, namely the \emph{Dini-mean-oscillation property}.

\subsection{General Settings of the Paper and the Main Result}
The general settings are milder than those given by Simon in \cite{Simon1}. Fix an index $\kappa\in\mathbb{N}_+^n$ with $\kappa_l>0\,\forall l$. Set the $\kappa$-(quasi-)distance as
$$|x-y|_\kappa:=\sqrt{\sum_{l=1}^n|x^l-y^l|^{2/\kappa_l}}.$$
The $\kappa$-ball centered at $x$ with radius $r$ will be denoted as $B^\kappa_r(x)$. Obviously $B^\kappa_r(x)$ is star-shaped with respect to $x$. Its boundary $\partial B^\kappa_r(x)$ is diffeomorphic to $\partial B(x,r)\approx S^{n-1}$ via the map $\partial B^\kappa_r(x)\ni y\to r(y-x)/|y-x|$, except on an algebraic subset of dimension at most $n-2$. For any $r>0$, define the anisotropic dilation operator $T_r$ by
$$T_r(x^1,\cdots,x^n)=(r^{\kappa_1}x^1,\cdots,r^{\kappa_n}x^n).$$
Then $T_r^{-1}=T_{1/r}$, $|T_rx|_\kappa=r|x|_\kappa$, $\det T_r=r^{|\kappa|}$. Hence
$$|B^{\kappa}_r(0)|=r^{|\kappa|}|B^\kappa_1(0)|.$$
Note that these constructions turn $(\mathbb{R}^n,|\cdot|_\kappa)$ together with the Lebesgue measure into a space of homogeneous type in the sense of Coifman and Weiss \cite{CW}.

Let $\mathcal{A},\mathcal{B}\subset\mathbb{N}_0^n$ (non-negative integers) be finite index sets such that for some fixed $m$,
$$(\alpha+\beta)\cdot\kappa=m$$
for all $\alpha\in\mathcal{A}$, $\beta\in\mathcal{B}$. Suppose $\mathcal{A}$ is \emph{$\mathcal{B}$-complete}, in the sense that $\mathcal{A}$ coincides with the collection of indices $\alpha\in\mathbb{N}_0^n$ such that $(\alpha+\beta)\cdot\kappa=m\,\forall\beta\in\mathcal{B}$. Define  $\mathcal{A}'$ to be the collection of $\gamma\in\mathbb{N}_0^n$ such that $(\gamma+\beta)\cdot\kappa\leq m\,\forall\beta\in\mathcal{B}$. Throughout the paper, the notation on indices will be fixed. We are also going to employ the following convention: for any domain $\Omega\subset\mathbb{R}^n$, if $\mathfrak{F}(\Omega)$ is a subspace of $\mathfrak{D}'(\Omega)$, then $\mathfrak{F}^\mathcal{A}(\Omega)$ will denote the collection of all $u\in\mathfrak{D}'(\Omega)$ such that $D^\alpha u\in\mathfrak{F}(\Omega)$ for all $\alpha\in\mathcal{A}$.

A partial differential operator $P$ of the form
\begin{equation}\label{Constprin}
P=\sum_{\alpha\in\mathcal{A},\beta\in\mathcal{B}}D^{\beta}(a_{\alpha\beta}D^\alpha)
\end{equation}
is said to be \emph{$\kappa$-homogeneous of order $m$}. For the moment, assume that all coefficients are (possibly complex) constants. The symbol of $P$ is
$$p(\xi)=\sum_{\alpha\in\mathcal{A},\beta\in\mathcal{B}}a_{\alpha\beta}(i\xi)^{\alpha+\beta},$$
and satisfies
\begin{equation}\label{kappahom}
p(T_r\xi)=r^{m}p(\xi).
\end{equation}
Assume that $P$ is \emph{hypoelliptic}; equivalently, this means $p(\xi)\neq0$ for $\xi\in\mathbb{R}^n\setminus\{0\}$, by standard criteria for hypoellipticity. Then
$$
|\xi|^m_\kappa\leq C\left|p\left(T^{-1}_{|\xi|_\kappa}\xi\right)\right||\xi|_\kappa^m=C|p(\xi)|,
$$
where $C=\left(\inf_{\partial B^\kappa_1(0)}|p|\right)^{-1}=C(m,n,\{a_{\alpha\beta}\})$. We thus require the following condition:
\begin{equation}\label{hypoell}
\lambda|\xi|^m_\kappa\leq |p(\xi)|\leq\Lambda|\xi|^m_\kappa.
\end{equation}
This slightly generalizes the notion of semi-elliptic differential operators, which is defined, for example, in \cite{Hor}, p.67. The main result is now stated as follows, which generalize the results given by Wang in \cite{Wang} and Dong, Kim in \cite{DK}.
\begin{theorem}\label{VariablePrin}
Let $u\in C^{\mathcal{A}}(B^\kappa_{R_0})$ solve the equation
$$Pu=\sum_{\beta\in\mathcal{B}}\sum_{\alpha\in\mathcal{A}}D^\beta(a_{\alpha\beta}D^\alpha u)=\sum_{\beta\in\mathcal{B}}D^\beta f_\beta.$$
Define
$$\bar\omega_a(\rho):=\sup_{B^\kappa_\rho(x_0)\Subset B^\kappa_{R_0}}\sum_{\alpha\in\mathcal{A},\beta\in\mathcal{B}}\dashint_{B^\kappa_\rho(x_0)}|a_{\alpha\beta}(x)-a_{\alpha\beta}(x_0)|dx,$$
$$\bar\omega_f(\rho):=\sup_{B^\kappa_\rho(x_0)\Subset B^\kappa_{R_0}}\sum_{\beta\in\mathcal{B}}\dashint_{B^\kappa_\rho(x_0)}|f_{\beta}(x)-f_{\beta}(x_0)|dx.$$
Suppose the following condition: for $p(x,\xi)=\sum_{\alpha\in\mathcal{A},\beta\in\mathcal{B}}a_{\alpha\beta}(x)(i\xi)^{\alpha+\beta}$, there are constants $0<\lambda<\Lambda$ such that for all $(x,\xi)\in\mathbb{R}^n\times\mathbb{R}^n$,
\begin{equation}\label{Varellip}
\lambda|\xi|_\kappa^m\leq|p(x,\xi)|\leq\Lambda|\xi|^m_\kappa.
\end{equation}
Then for any $\gamma,\theta\in(0,1)$, we have for all $\alpha\in\mathcal{A}$ that for $C=C(\lambda,\Lambda,\kappa,\gamma,\theta,B^\kappa_{R_0})$,
$$
\begin{aligned}
\sup_{\substack{x,y\in B^\kappa_{\theta R_0} \\ |x-y|_\kappa\leq r}}|&D^\alpha u(x)-D^\alpha u(y)|\\
&\leq C\left([|u|]_{\mathcal{A},\infty;B^\kappa_{R_0}}r^\gamma+
\int_0^{r}\frac{\bar\omega_f(\rho)}{\rho}d\rho+r^\gamma\int_{r}^{R_0}\frac{\bar\omega_f(\rho)}{\rho^{1+\gamma}}d\rho\right)\\
&\quad\quad\quad+C[|u|]_{\mathcal{A},\infty;B^\kappa_{R_0}}\left(
\int_0^{r}\frac{\bar\omega_a(\rho)}{\rho}d\rho+r^\gamma\int_{r}^{R_0}\frac{\bar\omega_a(\rho)}{\rho^{1+\gamma}}d\rho\right).
\end{aligned}
$$
\end{theorem}
Other interior estimates can be deduced as corollaries of this main result. If the coefficients happen to be constants, then a different result, which is obtained from a completely different harmonic-analytic argument, is valid:
$$
\begin{aligned}
&\sup_{\substack{x,y\in B^\kappa_{\theta R_0} \\ |x-y|_\kappa\leq r}}|D^\alpha u(x)-D^\alpha u(y)|\\
&\leq C\left([|u|]_{\mathcal{A},\infty;B^\kappa_{R_0}}+\sup_{B^\kappa_{R_0}}\sum_{\beta\in\mathcal{B}}|f_\beta|\right)r
+C\left(\int_0^{r}\frac{\omega_f(\rho)}{\rho}d\rho+r\int_{r}^{R_0}\frac{\omega_f(\rho)}{\rho^{2}}d\rho\right),
\end{aligned}
$$
Where $\omega$ denotes the usual oscillation. The advantage of these results is that they unify the case of strong and weak solutions, and reveal the link between the algebraic property of the principal symbol (i.e. $\kappa$-homogeneity) and the analytic property (i.e. modulus of continuity) of the solutions.

We point out that the assumption of $\mathcal{B}$-completeness is just for convenience considerations; furthermore, given index sets $\mathcal{A},\mathcal{B}$, the existence of a differential operator satisfying the hypoellipticity requirement actually implies structural results concerning the index sets $\mathcal{A},\mathcal{B}$ themselves. These structural results are obtained as corollaries of the interior estimates stated above. Let us present an anisotropic Sobolev-type theorem as an illustrative example:
\begin{theorem}\label{Sobolev}
Suppose $u\in W^{\mathcal{A},p}(B^\kappa_R)$. Fix a $\theta\in(0,1)$. If $1\leq p<|\kappa|$, then for all $\alpha'\in\mathcal{A}'$ and all $q\in[1,|\kappa|p/(|\kappa|-p))$, the following inequality is valid:
$$\|D^{\alpha'} u\|_{q;B^\kappa_{\theta R}}\leq C\sum_{\alpha\in\mathcal{A}}\|D^\alpha u\|_{p;B^\kappa_R},$$
where $C=C(p,q,\theta,\lambda,\Lambda,B^\kappa_R)$. If $p>|\kappa|$, then for all $\alpha'\in\mathcal{A}'$, $D^{\alpha'} u$ is continuous, and in fact for all $x,y\in B^\kappa_{\theta R}$, the following inequality is valid:
$$|D^{\alpha'} u(x)-D^{\alpha'} u(y)|\leq C|x-y|_\kappa^{1-p/|\kappa|}\sum_{\alpha\in\mathcal{A}}\|D^\alpha u\|_{p;B^\kappa_R},$$
where $C=C(p,\theta,\lambda,\Lambda,B^\kappa_R)$.
\end{theorem}

\section{Constant Coefficient Case}
\subsection{A Multiplier Theorem and Global Estimates}
To deduce estimates on the modulus of continuity in the constant coefficient case, we first present the following multiplier theorem.
\begin{theorem}\label{Mult}
Let $m$ be bounded and smooth on $\mathbb{R}^n\setminus\{0\}$. Assume that for all multi-indices $\gamma$,
\begin{equation}\label{MihHor}
\sup_{R>0}R^{-|\kappa|+\gamma\cdot\kappa}\int_{R<|\xi|_\kappa<2R}|D^\gamma m(\xi)|d\xi=A_\gamma<\infty.
\end{equation}

(A) The operator $f\to(m\hat{f})^\vee$ is of weak (1,1) type and strong $(p,p)$ type for $p\in(1,\infty)$.

(B) Define the oscillation of $f$ to be $\omega_f(\rho):=\sup_{|x-y|_\kappa\leq \rho}|f(x)-f(y)|$. If
$$\int_0^\delta\frac{\omega_f(\rho)}{\rho}d\rho<\infty,\,\int_r^{\infty}\frac{\omega_f(\rho)}{\rho^2}d\rho=o(r^{-1}),\,r\to0,$$
then $u:=(m\hat{f})^\vee$ is represented by a continuous function of polynomial growth, and there exists a polynomial $Q$ such that
\begin{equation}\label{normcont}
\begin{aligned}
|[u(x)-Q(x)]-&[u(y)-Q(y)]|\\
&\leq C\left(\int_{0}^{|x-y|_\kappa}\frac{\omega_f(\rho)}{\rho}d\rho
+|x-y|_\kappa\int_{|x-y|_\kappa}^\infty\frac{\omega_f(\rho)}{\rho^2}d\rho\right),
\end{aligned}
\end{equation}
where the constant $C$ only depends on $m$. If either $m$ is supported away from 0 or $f\in L^p(\mathbb{R}^n)$ with $p\in(1,\infty)$, then the polynomial $Q$ can be taken as zero.
\end{theorem}
\emph{Rmark.} Of course assumption (\ref{MihHor}) follows from a direct imitation of the Mihlin¨CH\"{o}rmander multiplier theorem. The proof is in fact quite similar. The result is in fact already covered by the main theorem in Peetre's paper \cite{P}.
\begin{proof}[Proof of theorem \ref{Mult}]
Let $\psi\in C_0^\infty(\mathbb{R}^n)$ be such that $0\leq\psi\leq1$, $\psi$ is supported on $B^\kappa_1(0)$ and equal to 1 on $B^\kappa_{1/2}(0)$, and let $\varphi(\xi)=\psi(T_2^{-1}\xi)-\psi(\xi)$. Then with $\varphi_j(\xi)=\varphi(T_2^{-j}\xi)$, there is a decomposition of 1 for $\xi\neq0$:
$$1=\sum_{j\in\mathbb{Z}}\varphi_j(\xi)=\sum_{j\in\mathbb{Z}}\varphi(T_2^{-j}\xi).$$
Each block is a compactly supported smooth function. Write $m_j=m\varphi_j$. Note that for all $j\in\mathbb{Z}$,
\begin{equation}\label{cancel}
\int_{\mathbb{R}^n}\check{m}_j=m\varphi_j(0)=0,\quad\int_{\mathbb{R}^n}D\check{m}_j=0\cdot m\varphi_j(0)=0.
\end{equation}
By assumption, each $m_j$ is compactly supported and smooth, and
$$\check{m}_j(x)=2^{j|\kappa|}\left[\varphi(\xi)m(T_2^j\xi)\right]^\vee(T_2^jx).$$
Let $(r,\theta)$ be a polar coordinate system with respect to $|\cdot|_\kappa$, i.e., $r(x)=|x|_\kappa$, and $\theta$ is a smooth coordinate on $B^\kappa(0,1)$ (which is defined and smooth on $\partial B^\kappa(0,1)$ except on an $(n-2)$-dimensional algebraic subset of $\partial B^\kappa(0,1)$). Let $\Gamma_\theta$ be the corresponding volume form on $\partial B^\kappa(0,1)$ (which is defined almost everywhere on $\partial B^\kappa(0,1)$), and let $J(r,\theta)$ be the metric determinant of $(r,\theta)$. A direct calculation gives
$$\int_{\partial B^\kappa_1(0)}J(r,\theta)\Gamma_\theta=|\kappa||B^\kappa_1(0)|r^{|\kappa|-1}.$$
Define a function $\eta=\eta_1+\eta_2$ on $[0,\infty)$, where
$$
\begin{aligned}
\eta_1(r):=\sup_{j\in\mathbb{Z}}\int_{\partial B^\kappa_1(0)}
\left|\left[\varphi(\xi)m(T_2^j\xi)\right]^\vee(r\theta)\right|J(r,\theta)\Gamma_\theta,
\end{aligned}
$$
$$
\eta_2(r):=\sup_{j\in\mathbb{Z}}\int_{\partial B^\kappa_1(0)}
\left|\left[\xi^l\varphi(\xi)m(T_2^j\xi)\right]^\vee(r\theta)\right|J(r,\theta)\Gamma_\theta.
$$
Let us compute, for any index $\gamma$, any $y\in\mathbb{R}^n$ and any $j\in\mathbb{Z}$,
$$
\begin{aligned}
\left|y^\gamma\left[\varphi(\xi)m(T_2^j\xi)\right]^\vee(y)\right|&\leq\int_{\mathbb{R}^n}|D^\gamma(\varphi(\xi)m(T_2^j\xi))|d\xi\\
&\leq C_\gamma\sum_{\delta\leq\gamma}\sup|D^{\gamma-\delta}\varphi|
\int_{1/2\leq|\xi|_\kappa\leq2}2^{j\kappa\cdot\delta}|(D^\delta m)(T_2^j\xi)|d\xi\\
&\leq C_\gamma\sum_{\delta\leq\gamma}
2^{-j|\kappa|+j\kappa\cdot\delta}\int_{2^{j-1}\leq|\xi|_\kappa\leq2^{j+1}}|(D^\delta m)(\xi)|d\xi\\
&\leq C_\gamma\sum_{\delta\leq\gamma}A_\delta
\end{aligned}
$$
by assumption (\ref{MihHor}). A similar calculation gives
$$\left|y^\gamma\left[\xi^l\varphi(\xi)m(T_2^j\xi)\right]^\vee(y)\right|\leq C_\gamma\sum_{\delta\leq\gamma}A_\delta.$$
Consequently, for any $k\geq0$, there is a constant $C_k$ independent of $j$ such that
$$r^k\left|\left[\varphi(\xi)m(T_2^j\xi)\right]^\vee(r\theta)\right|
+r^k\left|\left[\xi^l\varphi(\xi)m(T_2^j\xi)\right]^\vee(r\theta)\right|\leq C_k$$
for all $r\in(0,\infty)$, $\theta\in\partial B^\kappa_1(0)$. So $\eta(r)$ is bounded on $(0,\infty)$ and decays faster than any rational function at infinity. This implies that $\sum_{j\in\mathbb{Z}}\check{m}_j$ converges locally in $L^1(\mathbb{R}^n\setminus\{0\})$ to some $L^1_{\text{loc}}$ function $K$, satisfying
$$\sup_{R>0}\int_{R\leq|x|_\kappa\leq 2R}|K(x)|dx<\infty.$$
Note that $K$ coincides with $\check{m}$ on $\mathbb{R}^n\setminus\{0\}$.

(A) For the weak $(1,1)$ inequality, we quote a theorem on singular integrals over a space of homogeneous type; see \cite{CW}, pp.74-75. By this theorem, it suffices to verify the H\"{o}rmander type condition
$$\sup_{y\in\mathbb{R}^n}\int_{|x|_\kappa\geq 2|y|_\kappa}|K(x-y)-K(x)|dx<\infty.$$
The derivation of this inequality is exactly the same as in the standard proof of the Mihlin-H\"{o}rmander theorem, since we already obtained the growth estimate for $D\check{m}_j$. The strong $(p,p)$ inequality with $1<p\leq2$ follows from interpolation, and the case with $p>2$ follows from a dual argument.

(B) Let us put $u_j=(m_j\hat{f})^\vee$. Each $u_j$ is analytic and can be written as a convolution:
$$u_j(x)=\int_{\mathbb{R}^n}f(y)\check{m}_j(x-y)dy.$$
By the cancellation properties (\ref{cancel}), we compute
$$
\begin{aligned}
|u_j(x)|
&=\left|\int_{\mathbb{R}^n}\check{m}_j(x-y)[f(y)-f(x)]dy\right|\\
&\leq\int_{\mathbb{R}^n}|\check{m}_j(x-y)|\omega_f(|x-y|_\kappa)dy\\
&=\int_{\mathbb{R}^n}\left|\left[\varphi(\xi)m(T_2^j\xi)\right]^\vee(y)\right|\omega_f(2^{-j}|y|_\kappa)dy,
\end{aligned}
$$
$$
\begin{aligned}
|D_lu_j(x)|
&=\left|\int_{\mathbb{R}^n}D_l\check{m}_j(x-y)[f_\beta(y)-f_\beta(x)]dy\right|\\
&\leq\int_{\mathbb{R}^n}|D_l\check{m}_j(x-y)|\omega_f(|x-y|_\kappa)dy\\
&=2^{j\kappa_l}\int_{\mathbb{R}^n}\left|\left[\xi^l\varphi(\xi)m(T_2^j\xi)\right]^\vee(y)\right|\omega_f(2^{-j}|y|_\kappa)dy.
\end{aligned}
$$

So, for some constant $C$ depending on the numbers $\{A_\gamma\}$ only,
$$|u_j(x)|
\leq C\int_{0}^\infty\eta(r)\omega_f(2^{-j}r)dr,\,
|D_lu_j(x)|
\leq2^{j\kappa_l}C\int_{0}^\infty\eta(r)\omega_f(2^{-j}r)dr.$$
We compute
$$
\begin{aligned}
\sum_{l=1}^n|D_lu_j|_0|x^l-y^l|
&\leq C\sum_{l=1}^n2^{j\kappa_l}|x^l-y^l|\int_{0}^\infty\eta(r)\omega_f(2^{-j}r)dr\\
&\leq C\sum_{l=1}^n2^{j\kappa_l}|x-y|_{\kappa}^{\kappa_l}\int_{0}^\infty\eta(r)\omega_f(2^{-j}r)dr.
\end{aligned}
$$
So
$$
\begin{aligned}
\sum_{j\in\mathbb{Z}}|u_j(x)-u_j(y)|&\leq\sum_{j\in\mathbb{Z}}\min\left(2|u_j|_0,\sum_{l=1}^n|D_lu_j|_0|x^l-y^l|\right)\\
&\leq C\sum_{l=1}^n\sum_{j\in\mathbb{Z}}\min\left(1,2^{j\kappa_l}|x-y|_\kappa^{\kappa_l}\right)
\int_{0}^\infty\eta(r)\omega_f(2^{-j}r)dr.
\end{aligned}
$$
Splitting the sum into two parts with ${j>-\log_2|x-y|_\kappa}$ and ${j\leq-\log_2|x-y|_\kappa}$ respectively, we obtain
$$
\begin{aligned}
\sum_{j>-\log_2|x-y|_\kappa}&\leq C\sum_{l=1}^n\sum_{j>-\log_2|x-y|_\kappa}\int_{0}^\infty\eta(r)\omega_f(2^{-j}r)dr\\
&\leq C\sum_{j\geq0}\int_{0}^\infty\eta(r)\omega_f(2^{-j}|x-y|_\kappa r)dr,
\end{aligned}
$$
$$
\begin{aligned}
\sum_{j\leq-\log_2|x-y|_\kappa}&\leq\sum_{l=1}^n\sum_{j\leq-\log_2|x-y|_\kappa}2^{j\kappa_l}|x-y|_\kappa^{\kappa_l}
\int_{0}^\infty\eta(r)\omega_f(2^{-j}r)dr\\
&\leq C\sum_{l=1}^n\sum_{j\geq0}\int_{0}^\infty2^{-j\kappa_l}\eta(r)\omega_f(2^{j+1}|x-y|_\kappa r)dr.
\end{aligned}
$$
By the monotonicity property of $\omega_f$, we compute
$$
\begin{aligned}
\sum_{j>-\log_2|x-y|_\kappa}
&\leq C\int_{0}^\infty\eta(r)\sum_{j\geq0}(2^{-j+1}-2^{-j})\frac{\omega_f(2^{-j}|x-y|_\kappa r)}{2^{-j}}dr\\
&\leq C\int_{0}^\infty\eta(r)\int_0^{|x-y|_\kappa r}\frac{\omega_f(\rho)}{\rho}d\rho dr,
\end{aligned}
$$
$$
\begin{aligned}
\sum_{j\leq-\log_2|x-y|_\kappa}
&\leq C\sum_{l=1}^n\int_{0}^\infty\eta(r)\sum_{j\geq0}2^{\kappa_l}(2^{(j+2)\kappa_l}-2^{(j+1)\kappa_l})
\frac{\omega_f(2^{j+1}|x-y|_\kappa r)}{(2^{\kappa_l}-1)2^{2(j+1)\kappa_l}}dr\\
&\leq C\sum_{l=1}^n\int_{0}^\infty\eta(r)(|x-y|_\kappa r)^{2\kappa_l-1}\int_{|x-y|_\kappa r}^\infty
\frac{\omega_f(\rho)}{\rho^{2\kappa_l}}d\rho dr.
\end{aligned}
$$
By Fubini's theorem, we then compute
\begin{equation}\label{sum1}
\sum_{j>-\log_2|x-y|_\kappa}\leq C\int_{0}^\infty\frac{\omega_f(\rho)}{\rho}\int_{|x-y|_\kappa^{-1}\rho}^\infty\eta(r)drd\rho,
\end{equation}
\begin{equation}\label{sum2}
\sum_{j\leq-\log_2|x-y|_\kappa}\leq C\sum_{l=1}^n|x-y|^{2\kappa_l-1}_\kappa\int_{0}^\infty\frac{\omega_f(\rho)}{\rho^{2\kappa_l}}
\int_0^{|x-y|_\kappa^{-1}\rho}\eta(r)r^{2\kappa_l-1}drd\rho.
\end{equation}
Since $\eta$ is bounded and decays faster than any rational function, the right-hand-side of (\ref{sum1}) is controlled by
$$
\begin{aligned}
\int_{0}^{|x-y|_\kappa}&\frac{\omega_f(\rho)}{\rho}d\rho\int_0^\infty\eta(r)dr+C\int_{|x-y|_\kappa}^\infty\frac{\omega_f(\rho)}{\rho}
\int_{|x-y|_\kappa^{-1}\rho}^\infty\frac{dr}{r^2}d\rho\\
&\leq C\left(\int_{0}^{|x-y|_\kappa}\frac{\omega_f(\rho)}{\rho}d\rho
+|x-y|_\kappa\int_{|x-y|_\kappa}^\infty\frac{\omega_f(\rho)}{\rho^2}d\rho\right),
\end{aligned}
$$
and the right-hand-side of (\ref{sum2}) is controlled by
$$
\begin{aligned}
\sum_{l=1}^n&|x-y|_\kappa^{2\kappa_l-1}\left(\int_{0}^{|x-y|_\kappa}\frac{\omega_f(\rho)}{\rho^{2\kappa_l}}d\rho
\int_0^{{|x-y|_\kappa}^{-1}\rho}\eta(r)r^{2\kappa_l-1}dr\right.\\
&\left.\quad\quad\quad+\int_{|x-y|_\kappa}^\infty\frac{\omega_f(\rho)}{\rho^{2\kappa_l}}d\rho
\int_{0}^{\infty}\eta(r)r^{2\kappa_l-1}dr\right)\\
&\leq C\sum_{l=1}^n|x-y|_\kappa^{2\kappa_l-1}\left(\int_{0}^{|x-y|_\kappa}
\frac{\omega_f(\rho)}{\rho^{2\kappa_l}}\frac{\rho^{2\kappa_l-1}}{|x-y|_\kappa^{2\kappa_l-1}}d\rho
+\int_{|x-y|_\kappa}^\infty\frac{\omega_f(\rho)}{\rho^{2\kappa_l}}d\rho\right)\\
&\leq C\left(\int_{0}^{|x-y|_\kappa}\frac{\omega_f(\rho)}{\rho}d\rho
+|x-y|_\kappa\int_{|x-y|_\kappa}^\infty\frac{\omega_f(\rho)}{\rho^{2}}d\rho\right).
\end{aligned}
$$
By the assumptions on $\omega_f$, these are all finite integrals. Consequently, the series $\sum_{j\in\mathbb{Z}}[u_j(x)-u_j(y)]$ converges absolutely and locally uniformly on $\mathbb{R}^n\times\mathbb{R}^n$, to a continuous function of $(x,y)$. Set $y=0$; then by definition of the blocks $u_j$, the Fourier transform of the distribution $u-\sum_{j\in\mathbb{Z}}[u_j(x)-u_j(0)]$ is supported on zero, so there is a polynomial $Q$ such that
$$u=Q(x)+\sum_{j\in\mathbb{Z}}[u_j(x)-u_j(0)].$$
Hence, $u$ is represented by a continuous function of polynomial growth, and by absolute convergence of the series, the desired estimate (\ref{normcont}) on the modulus of continuity follows.

Finally, let us discuss when we can take $Q=0$. If either $m$ is supported away from zero or $f\in L^2(\mathbb{R}^n)$ (whence $\hat{f}\in L^2(\mathbb{R}^n)$), then obviously the series $\sum_{j\in\mathbb{Z}}\varphi_jm\hat{f}$ converges to $m\hat{f}=\hat{u}$ in $\mathcal{S}'(\mathbb{R}^n)$. So $\sum_{j\in\mathbb{Z}}u_j$ converges to $u$ in $\mathcal{S}'(\mathbb{R}^n)$, and this immediately legitimize us to take $Q=0$.
\end{proof}

Suppose now $u\in\mathcal{S}'(\mathbb{R}^n)$ solves a differential equation
\begin{equation}\label{globaleq}
Pu=\sum_{\alpha\in\mathcal{A},\beta\in\mathcal{B}}a_{\alpha\beta}D^{\alpha+\beta}u=\sum_{\beta\in\mathcal{B}}D^\beta f_\beta
\end{equation}
in the sense of distribution, where the symbol $p(\xi)$ of $P$ satisfies the general requirements (\ref{kappahom}) and (\ref{hypoell}), and the distributions $\{f_\beta\}_{\beta\in\mathcal{B}}$ are represented by continuous functions of polynomial growth. Taking Fourier transform, we obtain
$$p(\xi)\hat{u}(\xi)=\sum_{\beta\in\mathcal{B}}(i\xi)^\beta\hat{f}_\beta(\xi).$$
Consider the multiplier
$$m_{\alpha\beta}(\xi)=\frac{(i\xi)^{\alpha+\beta}}{p(\xi)},$$
where the index $\alpha$ is such that $(\alpha+\beta)\cdot\kappa=m$ for all $\beta\in\mathcal{B}$ (in particular, this is the case for $\alpha\in\mathcal{A}$). Let us verify that $m_{\alpha\beta}$ satisfies (\ref{MihHor}). In fact, using induction and the weighted arithmetic-geometric mean inequality, we obtian
$$|D^\gamma m(\xi)|\leq C|\xi|_\kappa^{-\gamma\cdot\kappa},$$
with $C=C(\alpha,\beta,\gamma,\kappa,\lambda,\Lambda)$. Set $v_\alpha=(\sum_{\beta\in\mathcal{B}}m_{\alpha\beta}\hat{f}_\beta)^\vee$; then obviously the Fourier transform of $(D^\alpha u-v_\alpha)$ is supported at the origin, so $(D^\alpha u-v_\alpha)$ is a polynomial. Applying the multiplier theorem \ref{Mult}, we obtain the following global estimate, which generalizes the results given by Simon \cite{Simon1} and Wang \cite{Wang}:
\begin{theorem}\label{Global}
Suppose $u\in\mathcal{S}'(\mathbb{R}^n)$ solves equation (\ref{globaleq}). Set
$$\omega_f(\rho):=\sup_{|x-y|_\kappa\leq \rho}\sum_{\beta\in\mathcal{B}}|f_\beta(x)-f_\beta(y)|.$$
If $\omega_f(\rho)/\rho\in L^1(0,1)$ and $\omega_f(\rho)/\rho^2\in L^1(1,+\infty)$, then for each $\alpha\in\mathbb{N}_0^n$ such that $(\alpha+\beta)\cdot\kappa=m\,\forall\beta\in\mathcal{B}$, $D^\alpha u$ is represented by a continuous function, and for some constant $C$ depending on $\lambda,\Lambda$ only, there is a polynomial $Q_\alpha$ such that
$$
\begin{aligned}
|[(D^\alpha u)&(x)-Q_\alpha(x)]-[(D^\alpha u)(y)-Q_\alpha(y)]|\\
&\leq C\left(\int_{0}^{|x-y|_\kappa}\frac{\omega_f(\rho)}{\rho}d\rho
+|x-y|_\kappa\int_{|x-y|_\kappa}^\infty\frac{\omega_f(\rho)}{\rho^2}d\rho\right).
\end{aligned}
$$
\end{theorem}

\subsection{Localization and Lower-order Derivatives}
In \cite{Simon1}, Simon assumed a regularity structure for the indices. To be precise, he assumed that for any $u\in C^\mathcal{A}(\Omega)$ with $\Omega$ being a convex domain, a Taylor type expansion
$$u(x)=\sum_{\gamma\in\mathcal{A}'\setminus\mathcal{A}}\frac{x^\gamma}{\gamma!}D^\gamma u(y)+\sum_{\gamma\in\mathcal{A}}\frac{x^\gamma}{\gamma!}D^\gamma u(x_\gamma),$$
where $x,y\in\Omega$ and $x_\gamma\in[x,y]$. However, if the index sets $\mathcal{A},\mathcal{B}$ already allow the existence of a hypoelliptic operator of the form (\ref{Constprin}), then it seems that such an assumption is removable.

Consider the differential operator
\begin{equation}\label{LowerConst}
P+H=\sum_{\alpha\in\mathcal{A},\beta\in\mathcal{B}}a_{\alpha\beta}D^{\alpha+\beta}
+\sum_{\substack{\alpha,\beta\in\mathbb{N}_0^n: \\ \kappa\cdot(\alpha+\beta)<m}}a_{\alpha\beta}D^{\alpha+\beta},
\end{equation}
and the corresponding symbol
\begin{equation}\label{LowerConstSym}
p(\xi)+h(\xi)=\sum_{\alpha\in\mathcal{A},\beta\in\mathcal{B}}a_{\alpha\beta}(i\xi)^{\alpha+\beta}
+\sum_{\substack{\alpha,\beta\in\mathbb{N}_0^n: \\ \kappa\cdot(\alpha+\beta)<m}}a_{\alpha\beta}(i\xi)^{\alpha+\beta}.
\end{equation}
The coefficients $\{a_{\alpha\beta}\}$ are still constants, and the principal symbol $p(\xi)$ still satisfies the general assumptions (\ref{kappahom}), (\ref{hypoell}). Then a direct computation gives $|h(\xi)|\leq C|\xi|_\kappa^{m-1}$, so $h(\xi)/p(\xi)=O(|\xi|_\kappa^{-1})$ as $|\xi|_\kappa\to\infty$. Consequently, there is a constant $C=C(\{a_{\alpha\beta}\})$ such that $|p(\xi)+h(\xi)|\geq \lambda|\xi|_\kappa^m/2$ if $|\xi|_\kappa\geq C$.

Without loss of generality, suppose that $u\in\mathfrak{D}'(B^\kappa_R)$, where $B^\kappa_R$ is the open $\kappa$-ball of radius $R$ centered at $0$, solves
\begin{equation}\label{localeq}
(P+H)u=\sum_{\beta\in\mathcal{B}}D^\beta f_\beta,
\end{equation}
where the distributions $f_\beta\in\mathfrak{D}'(B^\kappa_R)$ are represented by continuous functions. Given $\theta\in(0,1)$, let $\varphi\in C_0^\infty(B^\kappa_R)$ be a non-negative bump function such that $\varphi=1$ on a neighbourhood $\Omega\Subset B^\kappa_R$ of $\bar B^\kappa_{\theta R}$. To localize, one may employ a usual \emph{parametrix argument} as follows.

Define $v\in \mathcal{S}'(\mathbb{R}^n)$ to be the Fourier inverse transform of
$$\frac{1-\phi(\xi)}{p(\xi)+h(\xi)}\sum_{\beta\in\mathcal{B}}\widehat{D^\beta(\varphi f_\beta)}(\xi)
=\frac{1-\phi(\xi)}{p(\xi)+h(\xi)}\sum_{\beta\in\mathcal{B}}(i\xi)^\beta\widehat{\varphi f_\beta}(\xi),$$
where $\phi\in C_0^\infty(\mathbb{R}^n)$ is a non-negative bump function such that $\phi(\xi)=1$ when $|\xi|_\kappa\leq 2C$, and $\phi(\xi)=0$ for $|\xi|_\kappa\geq 4C$. Note that for each $\beta\in\mathcal{B}$, $\varphi f_\beta\in C_c(\mathbb{R}^n)$, so the Fourier transform of $D^\beta(\varphi f_\beta)$ is real-analytic. For any given $\alpha,\beta\in\mathbb{N}_0^n$ such that $(\alpha+\beta)\cdot\kappa\leq m$, the multiplier
$$\frac{1-\phi(\xi)}{p(\xi)+h(\xi)}(i\xi)^{\alpha+\beta}$$
is supported away from the origin, and satisfies the assumption of theorem \ref{Mult}. Hence for any $\alpha\in\mathcal{A}'$,
$$\|D^\alpha v\|_{p;\mathbb{R}^n}\leq C\sum_{\beta\in\mathcal{B}}\|f_\beta\|_{p;\mathbb{R}^n},$$
and for all $x,y\in\mathbb{R}^n$, with
$$\omega(\rho):=\sup_{|x-y|_\kappa\leq\rho}\sum_{\beta\in\mathcal{B}}|\varphi f_\beta(x)-\varphi f_\beta(y)|,$$
we have
$$
|D^\alpha v(x)-D^\alpha v(y)|\leq C\left(\int_{0}^{|x-y|_\kappa}\frac{\omega(\rho)}{\rho}d\rho
+|x-y|_\kappa\int_{|x-y|_\kappa}^\infty\frac{\omega(\rho)}{\rho^2}d\rho\right).
$$
Note that these estimates do not rely on any regularity-structural assumption on the index set $\mathcal{A}$.

Since $v\in L^2(\mathbb{R}^n)$ solves
$$(P+H)v=\sum_{\beta\in\mathcal{B}}D^\beta\left[\varphi f_\beta-(\varphi f_\beta)*\check{\phi}\right],$$
it follows that $w=u-v$ satisfies
$$(P+H)w=\sum_{\beta\in\mathcal{B}}D^\beta[(\varphi f_\beta)*\check{\phi}]$$
on $\Omega$, and the right-hand-side is the restriction of a Schwartz function on $\Omega$. We now present a Campanato type lemma, which is also used when deriving the results for equations with variable coefficients:

\begin{lemma}\label{Campanato1}
Suppose $w\in\mathfrak{D}'(B^\kappa_R)$ satisfies
$$(P+H)w=\sum_{\alpha\in\mathcal{A},\beta\in\mathcal{B}}a_{\alpha\beta}D^{\alpha+\beta}w=f,$$
where the operator $P$ still has constant coefficients and still satisfies the hypoellipticity condition (\ref{hypoell}), and $f\in C^\infty(B^\kappa_R)$. Then $w\in C^\infty(B^\kappa_R)$. Furthermore, for any $r\in(0,R/2]$ and any $p>0$, with $(w)_{0,r}$ being the mean of $w$ on $B^\kappa_r$, there are natural numbers $k=k(n,\kappa)$, $K=K(n,\kappa,m)$ such that for any $b\in\mathbb{C}$,
$$
\begin{aligned}
\|w-&(w)_{0,r}\|_{p;B^\kappa_r}\\
&\leq C\left(\frac{r}{R}\right)^{|\kappa|/p+\min\kappa_l}\|w-b\|_{p;B^\kappa_R}
+C(1+R^K)\sum_{j=0}^k\|D^jf\|_{1;B^\kappa_R},
\end{aligned}
$$
where the constant $C=C(p,\lambda,\Lambda,\kappa)$, provided that the right-hand-side is well-defined.
\end{lemma}
\begin{proof}
That $w\in C^\infty(B^\kappa_R)$ follows immediately from hypoellipticity.

For the moment, suppose $R=1$. We construct a family $\{\Omega_r\}_{r\in[3/4,1]}$ of star-shaped open sets, where each $\Omega_r$ is the dilation $T_r\Omega_1$, such that $\partial\Omega_1$ is smooth, and $B^\kappa_{1/2}\Subset\Omega_{3/4}\Subset\Omega_1\Subset B^\kappa_1$. For any bump function $\psi\in C_0^\infty(\Omega_{1})$, a direct computation with the aid of Leibniz formula gives
$$
\begin{aligned}
P(\psi w)&=\psi Pw+[P,\psi]w\\
&=\psi f+\sum_{\substack{\gamma\cdot\kappa<m, \\ (\gamma+\delta)\cdot\kappa<m}}c_{\gamma\delta}D^\gamma(D^\delta\psi w),
\end{aligned}
$$
where the constants $\{c_{\gamma\delta}\}$ depend $\{a_{\alpha\beta}\}$ only. Taking Fourier transform on both sides,
$$p(\xi)\widehat{\psi w}(\xi)=\widehat{\psi f}(\xi)+\sum_{\substack{\gamma\cdot\kappa<m, \\ (\gamma+\delta)\cdot\kappa<m}}c_{\gamma\delta}(i\xi)^\gamma\widehat{D^\delta\psi w}(\xi).$$
Multiplying both sides by $(1+|\xi|_\kappa^j)/(1+|\xi|_\kappa^m)$, using the weighted arithmetic-geometric mean inequality, the above equality gives
$$(1+|\xi|_\kappa^j)|\widehat{\psi w}(\xi)|\leq C(1+|\xi|_\kappa^{j-1})\left[|\widehat{\psi f}(\xi)|
+\sum_{\delta\cdot\kappa\leq m}|\widehat{D^\delta\psi w}(\xi)|\right],$$
where the constant $C=C(j,\lambda,\Lambda)$. Since $\psi$ is arbitrary, we might as well take $D^\delta\psi$ in place of $\psi$ and use induction on $j$, to obtain
\begin{equation}\label{temp1}
(1+|\xi|_\kappa^j)|\widehat{\psi w}(\xi)|
\leq C\left[(1+|\xi|_\kappa^{j-1})\sum_{\gamma\cdot\kappa\leq(j-1)m}|\widehat{D^\gamma\psi f}(\xi)|
+\sum_{\delta\cdot\kappa\leq jm}|\widehat{D^\delta\psi w}(\xi)|\right],
\end{equation}
where the constant $C=C(j,\lambda,\Lambda)$. Let $k$ be the smallest integer to allow the Sobolev embedding $W^{k,2}(\mathbb{R}^n)\hookrightarrow C^1(\mathbb{R}^n)$. Taking $j=k|\kappa|$ in (\ref{temp1}), we obtain
$$
|\xi|^k|\widehat{\psi w}(\xi)|\leq C\left[(1+|\xi|^{k|\kappa|})\sum_{\gamma\cdot\kappa\leq(k|\kappa|-1)m}|\widehat{D^\gamma\psi f}(\xi)|
+\sum_{\delta\cdot\kappa\leq k|\kappa|m}|\widehat{D^\delta\psi w}(\xi)|\right],
$$
where the constant $C=C(\lambda,\Lambda)$. Now let $\psi$ be such that $\psi=1$ on $\Omega_{1}$. Squaring and using the Plancherel theorem, this gives
\begin{equation}\label{temp4}
\|w\|_{k,2;\Omega_{3/4}}\leq C\left(\|w\|_{0,2;\Omega_{1}}+\|f\|_{k|\kappa|,2;\Omega_1}\right).
\end{equation}
Since $\partial\Omega_1$ is smooth, the Rellich¨CKondrachov theorem implies the compactness of the embedding $W^{j+1,s}(\Omega_1)\hookrightarrow W^{j,s}(\Omega_1)$ for all $s\geq1$. A standard blow-up argument then gives, for all $\varepsilon>0$, $p>0$ \footnote{If $p<1$ then $\|\cdot\|_p$ is not a norm, but $\|u\|_p\to0$ implies $u\to0$ in measure, so the usual blow-up argument is still valid.} and $u\in W^{j+1,s}(\Omega_1)$, an interpolation inequality
$$\|u\|_{j,s;\Omega_1}\leq\varepsilon\|u\|_{j+1,s;\Omega_1}+C_{\varepsilon,p,j,s}\|u\|_{p;\Omega_1},$$
where $j\geq k-1$. So (\ref{temp4}) is reduced to
$$\|w\|_{k,2;\Omega_{3/4}}\leq\varepsilon\|w\|_{k,2;\Omega_{1}}
+C_{\varepsilon,p,k}\|w\|_{p;\Omega_1}+C_{k}\sum_{j=0}^{2k|\kappa|}\|D^jf\|_{1;\Omega_1}.$$
Scaling and using a standard absorbtion argument (see for example the absorbtion lemma presented in \cite{Simon1}, section 4), we finally obtain
$$\|w\|_{C^{1}(B^{\kappa}_{1/2})}\leq C\|w\|_{k,2;B^{\kappa}_{1/2}}\leq C\left(\|w\|_{p;B^{\kappa}_{1}}+\sum_{j=0}^{2k|\kappa|}\|D^jf\|_{1;B^{\kappa}_{1}}\right),$$
where the constant $C=C(p,\lambda,\Lambda,k,|\kappa|)$.

Now fix $\theta\in(0,1/2]$. For any $y\in B^\kappa_{\theta}$, compute
\begin{equation}\label{temp2}
\begin{aligned}
\left(\dashint_{B^\kappa_\theta}|w(x)-(w)_{0,\theta}|^pdx\right)^{1/p}&\leq\text{osc}_{B^\kappa_\theta}w\\
&\leq\sup_{x,y\in B^\kappa_{\theta}}\|w\|_{C^1(B^\kappa_{1/2})}|x-y|\\
&\leq C\theta^{\min\kappa_l}\left(\|w\|_{p;B^{\kappa}_{1}}+\sum_{j=0}^{2k|\kappa|}\|D^jf\|_{1;B^{\kappa}_{1}}\right).
\end{aligned}
\end{equation}

To treat the general case $w\in C^\infty(B^\kappa_R)$ with $R\neq1$, we just have to consider the scaled function $w_R(x):=w(T_Rx)$, which is defined on $B^\kappa_1$, and satisfies $Pw_R=R^mf_R$ on $B^\kappa_R$. Replacing $w$ by $w_R$ in (\ref{temp2}), and setting $\theta=r/R$ with $r\leq R/2$, we obtain
$$
\begin{aligned}
&\left(\int_{B^\kappa_r}|w(x)-(w)_{0,r}|^pdx\right)^{1/p}\\
&\leq C\left(\frac{r}{R}\right)^{|\kappa|/p+\min\kappa_l}
\|w\|_{p;B^\kappa_R}+C(1+R^{km|\kappa|^2})\sum_{j=0}^{2k|\kappa|}\int_{B^\kappa_R}|D^jf|.
\end{aligned}
$$
Noticing $w-b$ is still a solution, we see that this is equivalent to the desired result.
\end{proof}

From the above reasonings, we immediately obtain the following theorem:
\begin{theorem}\label{Local}
Suppose that $u\in \mathfrak{D}'(B^\kappa_R)$ solves
\begin{equation}\label{localeq}
(P+H)u
=\sum_{\alpha\in\mathcal{A},\beta\in\mathcal{B}}a_{\alpha\beta}D^{\alpha+\beta}u
+\sum_{\substack{\alpha,\beta\in\mathbb{N}_0^n: \\ \kappa\cdot(\alpha+\beta)<m}}a_{\alpha\beta}D^{\alpha+\beta}u
=\sum_{\beta\in\mathcal{B}}D^\beta f_\beta,
\end{equation}
in the sense of distribution, where the symbol $p(\xi)$ of $P$ satisfies the general requirements (\ref{kappahom}) and (\ref{hypoell}).

(A) If the distributions $f_\beta\in L^p(B^\kappa_R)$ with $p\in (1,\infty)$, then $u\in W^{\mathcal{A},p}_{\mathrm{loc}}(B^\kappa_R)$, and for any $\theta\in(0,1)$ and any $\alpha\in\mathcal{A}'$,
$$\|D^\alpha u\|_{p;B^\kappa_{\theta R}}\leq C\left(\|u\|_{p;B^\kappa_R}+\sum_{\beta\in\mathcal{B}}\|f_\beta\|_{p;B^\kappa_R}\right),$$
where $C=C(p,\lambda,\Lambda,\theta,B^\kappa_R)$.

(B) If the distributions $f_\beta$ are represented by continuous functions and the modulus of continuity satisfy the conditions in theorem \ref{Global} in $B^\kappa_R$, then for any $\alpha\in\mathcal{A}'$, $D^\alpha u\in C(B^\kappa_R)$, and for any $\theta\in(0,1)$ and any $x,y\in B^\kappa_{\theta R}$,
$$
\begin{aligned}
|D^\alpha u(x)&-D^\alpha u(y)|\\
&\leq C|x-y|\left(\sup_{B^\kappa_R}|u|+\sum_{\beta\in\mathcal{B}}\sup_{B^\kappa_R}|f_\beta|\right)\\
&\quad+C\left(\int_{0}^{|x-y|_\kappa}\frac{\omega_f(\rho)}{\rho}d\rho
+|x-y|_\kappa\int_{|x-y|_\kappa}^\infty\frac{\omega_f(\rho)}{\rho^2}d\rho\right),
\end{aligned}
$$
where
$$\omega_f(\rho)=\sup_{|x-y|_\kappa\leq\rho}\sum_{\beta\in\mathcal{B}}|f_\beta(x)-f_\beta(y)|,$$
and $C=C(\lambda,\Lambda,\theta,B^\kappa_R)$.
\end{theorem}
As a corollary, we also obtain the following structural theorem for $W^{\mathcal{A},p}$ with $p>1$ and $C^\mathcal{A}$:
\begin{corollary}\label{Interpolation}
$u\in W^{\mathcal{A},p}_\mathrm{loc}(B^\kappa_R)$ implies $D^{\alpha'} u\in L^p_\mathrm{loc}(B^\kappa_R)$ for all $\alpha'\in\mathcal{A}'$, and $u\in C^{\mathcal{A}}(B^\kappa_R)$ implies $D^{\alpha'} u\in C(B^\kappa_R)$ for all $\alpha'\in\mathcal{A}'$.
\end{corollary}
A more delicate reasoning will give the anisotropic Sobolev-type theorem \ref{Sobolev}, whose proof is a delicate modification of the multiplier theorem \ref{Mult} and is sketched as follows. If $\alpha'\in\mathcal{A}'$, then the dyadic decomposition corresponding to the multiplier
$$m(\xi):=\frac{1-\phi(\xi)}{p(\xi)}\sum_{\beta\in\mathcal{B}}(i\xi)^{\alpha'+\beta}$$
(where $\phi(\xi)=1$ for $|\xi|_\kappa\leq4$) satisfies $m_j=0$ for $j\leq0$. Noting $(\alpha'+\beta)\cdot\kappa<m$, a direct computation gives
$$\|\check{m}_j\|_{s;\mathbb{R}^n}\leq C2^{[s-|\kappa|/(|\kappa|-1)]j}$$
for all $s<|\kappa|/(|\kappa|-1)$, so by Young's inequality, the operator $f\to(m\hat{f})^\vee$ maps $L^p(\mathbb{R}^n)$ to $L^{sp/(s-p)}(\mathbb{R}^n)$ for $1\leq p<|\kappa|$. As for $p>|\kappa|$, the proof for H\"{o}lder estimates is quite similar as in the standard Littlewood-Paley theory. The final result is obtained by a parametrix argument similar as above.

\section{Variable Coefficient Case}
We now turn to consider the variable coefficient case. The reasoning will be parallel as Dong and Kim's modification (\cite{DK}) of Campanato's argument (\cite{Camp}). The advantage of our approach is that we can investigate the equations in a unified manner without specifically developing a theory of existence for every particular type of equation, which differs from \cite{Camp}, \cite{DK} and \cite{Wang}.

\subsection{Campanato Type Lemmas}
Lemma \ref{Campanato1} and lemma \ref{Campanato2} presented in this subsection play the similar role as lemma 5.I and lemma 5.II in \cite{Camp}, or lemma 2.4 and lemma 2.5 in \cite{DEK}. The proofs strongly imitate \cite{DK} and \cite{DEK}.

For simplicity of the notation, we write $B^\kappa_\rho=B^\kappa_\rho(0)$ throughout this subsection. We are going to use the following notation: $(f)_{x_0,r}$ denotes the mean value of $f$ on $B^\kappa_r(x_0)$, and for $u\in W^{\mathcal{A},2}_{\text{loc}}(B^\kappa_R)$ and $x_0\in B^\kappa_R$, if $B^\kappa_r(x_0)\Subset B^\kappa_R$, then define
$$Q_{x_0,r}(u):=\sum_{\alpha\in\mathcal{A}}\frac{1}{\alpha!}(D^\alpha u)_{x_0,r}(x-x_0)^\alpha.$$
By the interpolation theorem \ref{Interpolation}, this polynomial is well-defined. For $\alpha,\alpha'\in\mathcal{A}$, a direct computation gives $D^\alpha(x-x_0)^{\alpha'}/\alpha'!\neq0$ if and only if $\alpha=\alpha'$. For any open set $\Omega$, we also define a (quasi-)norm $[|\cdot|]_{\mathcal{A},p;\Omega}$ for $p>0$ by
$$[|u|]^p_{\mathcal{A},p;\Omega}=\sum_{\alpha\in\mathcal{A}}\int_\Omega|D^\alpha u|^p,$$
with the obvious modification when $p=\infty$. Obviously the results presented here will hold with 0 replaced by any general $x_0$.

\begin{lemma}\label{Campanato2}
Fix a $p>0$ which does not equal 1. Suppose $u\in W^{\mathcal{A},1}_{\mathrm{loc}}(B^\kappa_R)\cap W^{\mathcal{A},p}_{\mathrm{loc}}(B^\kappa_R)$ satisfies
$$Pu=\sum_{\alpha\in\mathcal{A},\beta\in\mathcal{B}}a_{\alpha\beta}D^{\alpha+\beta}u=\sum_{\beta\in\mathcal{B}}D^\beta f_\beta,$$
where the operator $P$ still has constant coefficients and still satisfies the hypoellipticity condition (\ref{hypoell}), and $f_\beta\in L^1(B^\kappa_R)\cap L^p(B^\kappa_R)$ for all $\beta\in\mathcal{B}$.

(A) If $p>1$, then
$$
\begin{aligned}
&\left[|u-Q_{0,r}(u)|\right]_{\mathcal{A},p;B^\kappa_{r}}\\
&\leq C\left(\frac{r}{R}\right)^{|\kappa|/p+\min\kappa_l}
[|u-Q_{0,R}(u)|]_{\mathcal{A},p;B^\kappa_{R}}
+C(1+R^K)\sum_{\beta\in\mathcal{B}}\|f_\beta-(f_\beta)_{0,R}\|_{p;B^\kappa_{R}},
\end{aligned}
$$
where the constant $C=C(p,\lambda,\Lambda,\kappa)$, and the natural number $K=K(n,\kappa)$.

(B) If $0<p<1$, then with
$$\varphi_{\mathcal{A},p}(0,r)
:=\inf_{(b_\alpha)\in\mathbb{C}^\mathcal{A}}\sum_{\alpha\in\mathcal{A}}\left(\dashint_{B^\kappa_r}|D^\alpha w-b_\alpha|^p\right)^{1/p},$$
we have
$$\varphi_{\mathcal{A},p}(0,r)\leq C\left(\frac{r}{R}\right)^{\min\kappa_l}
\varphi_{\mathcal{A},p}(0,R)
+C(r^{-1}R)^{{|\kappa|}}(1+R^{2K})\sum_{\beta\in\mathcal{B}}\dashint_{B^\kappa_R}|f_\beta|,$$
where the constant $C=C(p,\lambda,\Lambda,\kappa)$, and the natural number $K=K(n,\kappa)$.
\end{lemma}
\begin{proof}
If $r\geq R/2$, then the inequality follows from an easy calculation. So it suffices to consider $r<R/2$.

As in the previous section, we are still going to use a parametrix argument, since it is in general impossible to construct a solution of $Pv=f$ such that $v$ is controlled in terms of $f$. Define
$$v=\left(\frac{1-\phi(\xi)}{p(\xi)}\sum_{\beta\in\mathcal{B}}(i\xi)^\beta\hat{f}_\beta(\xi)\right)^\vee,$$
where the $f_\beta$'s are extended to be zero outside $B^\kappa_R$, and $\phi\in C_0^\infty(\mathbb{R}^n)$ is a bump function with range $[0,1]$ which equals 1 on some neighbourhood of 0. Write $w=u-v$.

(A) It is easily verified that for any $\alpha\in\mathcal{A}$, the multiplier
$$\frac{1-\phi(\xi)}{p(\xi)}\sum_{\beta\in\mathcal{B}}(i\xi)^{\alpha+\beta}$$
meets all the requirements of theorem \ref{Mult}. So for $p>1$, there is a constant $C=C(n,p,\lambda,\Lambda,\phi)$ such that
$$\sum_{\alpha\in\mathcal{A}}\|D^\alpha v\|_{p;\mathbb{R}^n}\leq C\sum_{\beta\in\mathcal{B}}\|f_\beta\|_{p;B^\kappa_R}.$$

The distribution $w=u-v\in\mathfrak{D}'(B^\kappa_R)$ satisfies
$$Pw=\sum_{\beta\in\mathcal{B}}f_\beta*(D^\beta\check{\phi}).$$
The right-hand-side is the restriction of a Schwartz function on $B^\kappa_R$, the $L^1$ norm of whose $j$'th derivative is less than $$C(j,\phi)\sum_{\beta\in\mathcal{B}}\|f_\beta\|_{1;B^\kappa_R}.$$
By hypoellipticity, $w\in C^\infty(B^\kappa_R)$. For any $\alpha\in\mathcal{A}$, the function $D^\alpha w$ satisfies
$$PD^\alpha w=\sum_{\beta\in\mathcal{B}}f_\beta*(D^{\alpha+\beta}\check{\phi}).$$
Taking $a=(D^\alpha w)_{0,R}$, $f=\sum_{\beta\in\mathcal{B}}f_\beta*(D^{\alpha+\beta}\check{\phi})$ in lemma \ref{Campanato1}, we obtain
$$
\begin{aligned}
\|&D^\alpha w-(D^\alpha w)_{0,r}\|_{p;B^\kappa_r}\\
&\leq C\left(\frac{r}{R}\right)^{|\kappa|/p+\min\kappa_l}
\|D^\alpha w-(D^\alpha w)_{0,R}\|_{p;B^\kappa_R}+C(1+R^{K})\sum_{\beta\in\mathcal{B}}\|f_\beta\|_{1;B^\kappa_R}.
\end{aligned}
$$
Summing over $\alpha\in\mathcal{A}$, this gives
$$
\begin{aligned}
&\left[|w-Q_{0,r}(w)|\right]_{\mathcal{A},p;B^\kappa_{r}}\\
&\quad\leq C\left(\frac{r}{R}\right)^{|\kappa|/p+\min\kappa_l}
[|w-Q_{0,R}(w)|]_{\mathcal{A},p;B^\kappa_{R}}
+C(1+R^K)\sum_{\beta\in\mathcal{B}}\|f_\beta\|_{1;B^\kappa_R}.
\end{aligned}
$$

Now since the operator $u\to Q_{0,r}(u)$ is linear, and
$$\|D^\alpha v-(D^\alpha v)_{0,r}\|_{p;B^\kappa_r}\leq 2^{p-1}\|D^\alpha v\|_{p;B^\kappa_r},$$
we compute
$$
\begin{aligned}
\left[|u\right.&\left.-Q_{0,r}(u)|\right]_{\mathcal{A},p;B^\kappa_{r}}\\
&\leq[|w-Q_{0,r}(w)|]_{\mathcal{A},p;B^\kappa_{r}}+2^p[|v|]_{\mathcal{A},p;B^\kappa_{r}}\\
&\leq C\left(\frac{r}{R}\right)^{|\kappa|/p+\min\kappa_l}
[|w-Q_{0,R}(w)|]_{\mathcal{A},p;B^\kappa_{R}}
+C(1+R^K)\sum_{\beta\in\mathcal{B}}\|f_\beta\|_{1;B^\kappa_R}\\
&\leq C\left(\frac{r}{R}\right)^{|\kappa|/p+\min\kappa_l}
[|u-Q_{0,R}(u)|]_{\mathcal{A},p;B^\kappa_{R}}
+C(1+R^K)\sum_{\beta\in\mathcal{B}}\|f_\beta\|_{p;B^\kappa_R}.
\end{aligned}
$$
In order to obtain the desired result, just replace $u$ by $\tilde{u}=u-Q_{0,R}(u)$ in the above inequality. If $\mathcal{B}=\{0\}$, then $Pu=f$ a.e., so $P\tilde{u}=f-(f)_{0,R}$. If $\mathcal{B}\neq\{0\}$, then none of the indices in $\mathcal{B}$ could be 0 (since otherwise $\alpha\cdot\kappa=m\,\forall\alpha\in\mathcal{A}$, implying $\mathcal{B}=\{0\}$, a contradiction), so
$$P\tilde{u}=Pu=\sum_{\beta\in\mathcal{B}}D^\beta f_\beta=\sum_{\beta\in\mathcal{B}}D^\beta(f_\beta-(f)_{0,R}).$$
Noticing that for all $\alpha\in\mathcal{A}$,
$$D^\alpha(\tilde{u}-Q_{0,r}(\tilde{u}))=D^\alpha\tilde{u}-(D^\alpha\tilde{u})_{0,r}
=D^\alpha{u}-(D^\alpha{u})_{0,r}=D^\alpha({u}-Q_{0,r}({u})),$$
the desired result follows.

(B) Just as in (A), there is a constant $C=C(n,\lambda,\Lambda,\phi)$ such that for all $t\geq0$,
$$\sum_{\alpha\in\mathcal{A}}\left|\left\{x\in\mathbb{R}^n:|D^\alpha v(x)|\geq t\right\}\right|\leq\frac{C}{t}\sum_{\beta\in\mathcal{B}}\|f_\beta\|_{1;B^\kappa_R}.$$
So for $0<p<1$,
$$
[|v|]_{\mathcal{A};p,B^\kappa_\rho}^p
=\int_0^\infty pt^{p-1}\sum_{\alpha\in\mathcal{A}}\left|\left\{x\in B^\kappa_\rho:|D^\alpha v(x)|\geq t\right\}\right|dt\\
=\int_0^{\tau}+\int_{\tau}^\infty,
$$
where $\tau$ is to be determined. The first integral is bounded by $\tau^p|B^\kappa_r|$, and the second integral is controlled by the weak (1,1) inequality. So we obtain
$$
[|v|]_{\mathcal{A},p;B^\kappa_\rho}^p\leq\tau^p|B^\kappa_\rho|+\frac{C_p}{1-p}\tau^{p-1}\sum_{\beta\in\mathcal{B}}\|f_\beta\|_{1;B^\kappa_R}.
$$
Taking $\tau$ so that it minimizes the right-hand-side, we obtain
$$[|v|]_{\mathcal{A},p;B^\kappa_\rho}^p\leq C_p|B^\kappa_\rho|^{1-p}\left(\sum_{\beta\in\mathcal{B}}\|f_\beta\|_{1;B^\kappa_R}\right)^p.$$
Note that $\|\cdot\|_p$ is quasi-subadditive when $0<p<1$. For any $(b_\alpha)\in\mathbb{C}^\mathcal{A}$, with the aid of lemma \ref{Campanato1} and a similar computation as in the case $p>1$,
$$
\begin{aligned}
&\varphi_{\mathcal{A},p}(0,r)\leq\left(\dashint_{B^\kappa_r}|D^\alpha u-(D^\alpha w)_{0,r}|^p\right)^{1/p}\\
&\leq C_p\left(\dashint_{B^\kappa_r}|D^\alpha w-(D^\alpha w)_{0,r}|^p\right)^{1/p}+C_p\left(\dashint_{B^\kappa_r}|D^\alpha v|^p\right)^{1/p}\\
&\leq C\left(\frac{r}{R}\right)^{\min\kappa_l}
\sum_{\alpha\in\mathcal{A}}\left(\dashint_{B^\kappa_{R}}|D^\alpha w-b_\alpha|^p\right)^{1/p}
+Cr^{-{|\kappa|}}(1+R^K)\sum_{\beta\in\mathcal{B}}\|f_\beta\|_{1;B^\kappa_R}\\
&\leq C\left(\frac{r}{R}\right)^{\min\kappa_l}
\sum_{\alpha\in\mathcal{A}}\left(\dashint_{B^\kappa_{R}}|D^\alpha u-b_\alpha|^p\right)^{1/p}
+C(r^{-1}R)^{|\kappa|}(1+R^{2K})\sum_{\beta\in\mathcal{B}}\dashint_{B^\kappa_R}|f_\beta|.
\end{aligned}
$$
Taking infimum over all $(b_\alpha)\in\mathbb{C}^\mathcal{A}$, we obtain the desired result.
\end{proof}

\subsection{Freezing Coefficients}
We finally turn to the coefficient freezing step. We first present the following version of the absorbtion lemma:
\begin{lemma}\label{absorbtion}
Let $\varphi,\psi$ be non-negative measurable functions defined on $(0,R_0]$. Suppose that $\phi,\psi$ has the following quasi-monotone property: there are constants, $\tau\in(0,1),A\geq1$ such that for $r\in[\tau R,R]$ with $R\in(0,R_0]$, the inequality
$$A^{-1}\varphi(\tau R)\leq\varphi(r)\leq A\varphi(R),A^{-1}\psi(\tau R)\leq\psi(r)\leq A\psi(R)$$
holds. Suppose also for some constant $\gamma\in(0,1]$,
$$\varphi(\tau R)\leq\tau^\gamma\varphi(R)+\psi(R)$$
For all $R\in(0,R_0]$. Then for $r\in(0,R_0]$,
$$\varphi(r)\leq C\left(\frac{r}{R_0}\right)^\gamma\varphi(R_0)+Cr^{\gamma}\int_{r}^{R_0}\frac{\psi(\rho)}{\rho^{\gamma+1}}d\rho,$$
where the constant $C$ depends on $\tau,A$.
\end{lemma}
\begin{proof}
By iteration and using the quasi-monotone property, we obtain, for all tuples $\{\rho_j\}_{j=0}^k$ with $\rho_j\in[\tau^{j},\tau^{j-1}]$,
$$
\begin{aligned}
\varphi(\tau^{k+1} R_0)&\leq \tau^{(k+1)\gamma}\varphi(R_0)+\tau^{k\gamma}\sum_{j=0}^k\frac{\psi(\tau^jR_0)}{\tau^{j\gamma}}\\
&\leq\tau^{(k+1)\gamma}\varphi(R_0)+A\tau^{k\gamma}\sum_{j=0}^k\frac{\psi(\rho_jR_0)}{\tau^{j\gamma}}.
\end{aligned}
$$
Taking infimum for all tuples $\{\rho_j\}_{j=0}^k$ with $\rho_j\in[\tau^{j},\tau^{j-1}]$, we see that the right-hand-side is bounded from above by
$$\tau^{(k+1)\gamma}\varphi(R_0)+C(\tau^kR_0)^\gamma\int_{\tau^kR_0}^{R_0}\frac{\psi(\rho)}{\rho^{\gamma+1}}d\rho,$$
with $C$ depending on $\tau,A$. Now let $k$ be such that $\tau^{k+1}R_0<r\leq\tau^kR_0$. The desired result follows immediately.
\end{proof}

The interior estimate stated in theorem \ref{VariablePrin} employs lemma \ref{Campanato1}.
\begin{proof}[Proof of theorem \ref{VariablePrin}.]
It suffices to prove for $\theta<1/2$, since once this is done, the general case follows from a finite-covering argument. Fix any $x_0\in B^\kappa_{\theta R_0}$ and let $R\leq R_0/2$. Then for any $x_0\in B^\kappa_{\theta R_0}$, $B^\kappa_R(x_0)\Subset B^\kappa_{R_0}$. Set
$$P_0=\sum_{\alpha\in\mathcal{A},\beta\in\mathcal{B}}a_{\alpha\beta}(x_0)D^{\alpha+\beta},$$
and rewrite the equation as
$$P_0u=\sum_{\beta\in\mathcal{B}}\sum_{\alpha\in\mathcal{A}}D^\beta(a_{\alpha\beta}(x_0)-a_{\alpha\beta})D^\alpha u)
+\sum_{\beta\in\mathcal{B}}D^\beta f_\beta=:\sum_{\beta\in\mathcal{B}}D^\beta F_\beta.$$
If $\mathcal{B}=\{0\}$, then rewrite it as
$$P_0(u-Q(u))=\sum_{\alpha\in\mathcal{A}}(a_{\alpha\beta}(x_0)-a_{\alpha\beta})D^\alpha u
+(f-f(x_0))=:F$$
where $Q(u)=\sum_{\alpha\in\mathcal{A}}1/\alpha!D^\alpha u(x_0)(x-x_0)^\alpha$. Fix $0<p<1$, and let
$$\phi_{\mathcal{A},p}(x_0,r)
:=\inf_{(b_\alpha)\in\mathbb{C}^\mathcal{A}}\sum_{\alpha\in\mathcal{A}}\left(\dashint_{B^\kappa_r(x_0)}|D^\alpha w-b_\alpha|^p\right)^{1/p}.$$
By lemma \ref{Campanato2}, in either case $\mathcal{B}\neq\{0\}$ and $\mathcal{B}=\{0\}$, for $r\in(0,R]$,
\begin{equation}\label{temp3}
\begin{aligned}
\varphi&_{\mathcal{A},p}(x_0,r)\\
&\leq C\left(\frac{r}{R}\right)^{\min\kappa_l}
\varphi_{\mathcal{A},p}(x_0,R)
+C(r^{-1}R)^{|\kappa|}(1+R^{2K})\sum_{\beta\in\mathcal{B}}\dashint_{B^\kappa_R(x_0)}|F_\beta|\\
&\leq C\left(\frac{r}{R}\right)^{\min\kappa_l}
\varphi_{\mathcal{A},p}(x_0,R)\\
&\quad+C\left(\frac{r}{R}\right)^{-{|\kappa|}}(1+R_0^{2K})[|u|]_{\mathcal{A},\infty;B^\kappa_{B^{R_0}}}\bar\omega_a(R)
+C\left(\frac{r}{R}\right)^{-{|\kappa|}}(1+R_0^{2K})\bar\omega_f(R).\\
\end{aligned}
\end{equation}
Setting
$$\varphi(\rho)=\varphi_{\mathcal{A},p}(x_0,\rho),$$
and
$$\psi(\rho)=C\tau^{-{|\kappa|}}(1+R_0^{2K})[|u|]_{\mathcal{A},\infty;B^\kappa_{R_0}}\bar\omega_a(\rho)
+C\tau^{-{|\kappa|}}(1+R_0^{2K})\bar\omega_f(\rho),$$
we see that both $\varphi$ and $\psi$ match the requirements of lemma \ref{absorbtion} on $(0,R_0]$, and satisfy
$$\varphi(\tau R)\leq \tau^\gamma\varphi(R)+A\psi(R)$$
for all $R\in(0,R_0/2]$, where $\tau\in(0,1)$ is such that $\tau^{\min\kappa_l-\gamma}C\leq1$ with $C$ being the constant on the right-hand-side of (\ref{temp3}), and $A=A(\lambda,\Lambda,\gamma,R_0)$. By lemma \ref{absorbtion},
$$
\begin{aligned}
\varphi_{\mathcal{A},p}(x_0,r)&\leq C\left(\frac{r}{R_0}\right)^\gamma\varphi_{\mathcal{A},p}(x_0,R_0)+Cr^{\gamma}\int_{r}^{R_0}\frac{\psi(\rho)}{\rho^{\gamma+1}}d\rho\\
&\leq C[|u|]_{\mathcal{A},\infty;B^\kappa_{R_0}}r^\gamma+Cr^{\gamma}\int_{r}^{R_0}\frac{\psi(\rho)}{\rho^{\gamma+1}}d\rho.
\end{aligned}
$$
where the constant $C=C(\lambda,\Lambda,\kappa,\gamma,R_0)$. We now follow the reasoning as in the proof of theorem 1.5 in \cite{DK}. It follows that
$$
\begin{aligned}
\sup_{\substack{x,y\in B^\kappa_{\theta R} \\ |x-y|_\kappa\leq r}}|&D^\alpha u(x)-D^\alpha u(y)|\\
&\leq C\left([|u|]_{\mathcal{A},\infty;B^\kappa_{R}}r^\gamma+
\int_0^{r}r'^{\gamma-1}\int_{r'}^{R_0}\frac{\bar\omega_f(\rho)}{\rho^{1+\gamma}}d\rho dr'\right)\\
&\quad\quad\quad+C[|u|]_{\mathcal{A},\infty;B^\kappa_{R}}\left(
\int_0^{r}r'^{\gamma-1}\int_{r'}^{R_0}\frac{\bar\omega_a(\rho)}{\rho^{1+\gamma}}d\rho dr'\right).
\end{aligned}
$$
Finally, by Fubini's theorem, we compute
$$\int_0^rr'^{\gamma-1}\int_{r'}^{R_0}\frac{\psi(\rho)}{\rho^{1+\gamma}}d\rho dr'=\gamma^{-1}\int_0^r\frac{\psi(\rho)}{\rho}d\rho
+\gamma^{-1}r^\gamma\int_{r^\gamma}^{R_0}\frac{\psi(\rho)}{\rho^{1+\gamma}}d\rho.$$
Summing up, we obtain the desired result. Finally, just cover $B^\kappa_{\theta R_0}$ by finitely may balls $B^\kappa_R(x_0)\Subset B^\kappa_{R_0}$, and the previous argument applies.
\end{proof}

We now remove the assumption $u\in C^{\mathcal{A}}$. To obtain this, we need a hypoellipticity lemma:
\begin{lemma}\label{hypoellipticvar}
Consider the differential operator
$$P=\sum_{\beta\in\mathcal{B}}\sum_{\alpha\in\mathcal{A}}D^\beta(a_{\alpha\beta}D^\alpha),$$
where the coefficients are smooth on a domain $\Omega$, and the symbol $p(x,\xi)$ satisfies condition (\ref{Varellip}). Then $P$ is hypoelliptic.
\end{lemma}
To prove this lemma, we first deduce a G{\aa}rding type inequality for the operator, and then use mollification. The proof is quite standard, so we omit it here.

\begin{corollary}
Suppose the same regularity conditions on the coefficients and right-hand-side of the equation
$$Pu=\sum_{\beta\in\mathcal{B}}\sum_{\alpha\in\mathcal{A}}D^\beta(a_{\alpha\beta}D^\alpha u)=\sum_{\beta\in\mathcal{B}}D^\beta f_\beta$$
as in theorem \ref{VariablePrin}, but only assume $u\in W^{\mathcal{A},1}(B^\kappa_{R_0})$. Then in fact $u\in C^\mathcal{A}(B^\kappa_{R_0})$, and
$$[|u|]_{\mathcal{A},\infty;B^\kappa_{\theta R_0}}\leq C\left([|u|]_{\mathcal{A},1;B^\kappa_{R_0}}
+\int_0^{r}\frac{\bar\omega_f(\rho)}{\rho}d\rho+r^\gamma\int_{r}^{R_0}\frac{\bar\omega_f(\rho)}{\rho^{1+\gamma}}d\rho\right),$$
and the same estimate on the modulus of continuity for $\{D^\alpha u\}_{\alpha\in\mathcal{A}}$ holds.
\end{corollary}
\begin{proof}
Extend $a_{\alpha\beta}$ and $f_\beta$ to be zero outside $B^\kappa_{R_0}(x_0)$. Fix some compact subset $K\Subset B^\kappa_{R_0}(x_0)$. Let
$$P^{(\sigma)}=\sum_{\beta\in\mathcal{B}}\sum_{\alpha\in\mathcal{A}}D^\beta(a^{(\sigma)}_{\alpha\beta}D^\alpha ),$$
where $a_{\alpha\beta}^{(\sigma)}$ is the mollification of scale $\sigma$, and let
$$\tilde P^{(\sigma)}=\sum_{\beta\in\mathcal{B}}\sum_{\alpha\in\mathcal{A}}D^\beta(\tilde a^{(\sigma)}_{\alpha\beta}D^\alpha ),$$
where $\tilde a^{(\sigma)}_{\alpha\beta}$ equals $a^{(\sigma)}_{\alpha\beta}$ on some neighbourhood $\Omega\Subset B^\kappa_{R_0}(x_0)$ of $K$, and equals $a_{\alpha\beta}^{(\sigma)}(x_0)$ outside this neighbourhood. Further, let
$$P_0^{(\sigma)}=\sum_{\beta\in\mathcal{B}}\sum_{\alpha\in\mathcal{A}}D^\beta(a^{(\sigma)}_{\alpha\beta}(x_0)D^\alpha ).$$
Its symbol will be denoted as $p_0^{(\sigma)}(\xi)$. Then as $\sigma\to0$, $a^{(\sigma)}_{\alpha\beta}$ satisfy a uniform hypoelliticity condition (\ref{hypoell}) on $\Omega$, and their modulus of continuity are controlled by $\omega_a$. Similarly, the modulus of continuity of $f_\beta^{(\sigma)}$ are controlled by $\omega_f$. Furthermore, $a^{(\sigma)}_{\alpha\beta}$ and $f_\beta^{(\sigma)}$ are uniformly bounded and converge almost everywhere to $a_{\alpha\beta}$ and $f_\beta$ respectively.

Consider the operator
$$
\begin{aligned}
L_\sigma:v&\to \left(\frac{1-\phi(\xi)}{p_0^{(\sigma)}(\xi)}[(P^{(\sigma)}_0-\tilde P^{(\sigma)})v]^\wedge(\xi)\right)^\vee\\
&=\sum_{\beta\in\mathcal{B}}\left(\frac{[1-\phi(\xi)](i\xi)^\beta}{p_0^{(\sigma)}(\xi)}
\left[\sum_{\alpha\in\mathcal{A}}(a_{\alpha\beta}^{(\sigma)}(x_0)-\tilde a_{\alpha\beta}^{(\sigma)})D^\alpha v\right]^\wedge(\xi)\right)^\vee
\end{aligned}
$$
form $W^{\mathcal{A},2}(\mathbb{R}^n)$ to itself, where $\phi\in C_0^\infty(\mathbb{R}^n)$ equals 1 in a neighbourhood of 0. By the Plancherel theorem, it is easily checked that for every $\beta\in\mathcal{B}$, the multiplier
$$M^{(\sigma)}:v\to\left(\frac{[1-\phi(\xi)]}{p_0^{(\sigma)}(\xi)}\hat{v}(\xi)\right)^\vee$$
maps $\sum_{\beta\in\mathcal{B}}D^\beta L^2(\mathbb{R}^n)$ continuously to $W^{\mathcal{A},2}(\mathbb{R}^n)$ with norm independent of $\sigma$, so by assuming $R_0$ small enough (depending on $\bar\omega_a$ only), one might as well suppose the $(a_{\alpha\beta}^{(\sigma)}(x_0)-\tilde a_{\alpha\beta}^{(\sigma)})$'s are so close to zero that the norm of $L_{\sigma}$ is less than $1/2$. This legitimates us to define a $v_\sigma\in W^{\mathcal{A},2}(\mathbb{R}^n)$ by
$$
v_\sigma=(1-L_\sigma)^{-1}M^{(\sigma)}\left(P^{(\sigma)}u-\sum_{\beta\in\mathcal{B}}D^\beta f_\beta^{(\sigma)}\right).
$$
Then it satisfies
$$
v_\sigma=L_\sigma v_\sigma+M^{(\sigma)}\sum_{\beta\in\mathcal{B}}D^\beta(f_\beta-f_\beta^{(\sigma)})+M^{(\sigma)}[(P^{(\sigma)}-P)u],
$$
and obviously $\|v_\sigma\|_{\mathcal{A},2;\mathbb{R}^n}\to0$ as $\sigma\to0$. Choose a subsequence $v_{\sigma_k}$ such that $v_{\sigma_k}$ and $D^\alpha v_{\sigma_k}$ converges to 0 a.e..

Now a direct computation gives that in $\Omega$,
$$
\begin{aligned}
P^{(\sigma)}_0v_\sigma&=[P^{(\sigma)}_0v_\sigma-P^{(\sigma)}v_\sigma]-\check{\phi}*[P^{(\sigma)}_0v_\sigma-P^{(\sigma)}v_\sigma]\\
&\quad-\sum_{\beta\in\mathcal{B}}D^\beta f^{(\sigma)}_\beta+\sum_{\beta\in\mathcal{B}}D^\beta(\check{\phi}*f^{(\sigma)}_\beta)+P^{(\sigma)}u-\check{\phi}*(P^{(\sigma)}u),
\end{aligned}
$$
or
$$
\begin{aligned}
P^{(\sigma)}(u-v_\sigma)&=\sum_{\beta\in\mathcal{B}}D^\beta f^{(\sigma)}_\beta\\
&\quad+\check{\phi}*[P^{(\sigma)}_0v_\sigma-P^{(\sigma)}v_\sigma]+\check{\phi}*\left(P^{(\sigma)}u
-\sum_{\beta\in\mathcal{B}}D^\beta f^{(\sigma)}_\beta\right).
\end{aligned}
$$
By the hypoellipticity lemma (\ref{hypoellipticvar}), we see that $u-v_\sigma$ is smooth, and the modulus of continuity of the derivatives $\{D^\alpha(u-v_\sigma)\}_{\alpha\in\mathcal{A}}$ are estimated as in theorem \ref{VariablePrin}. The reasoning in the proof of theorem 1.5 of \cite{DK} gives the desired estimate for the upper bound of $\{D^\alpha(u-v_\sigma)\}_{\alpha\in\mathcal{A}}$. Now by standard properties of mollification operators, $D^\alpha(u-v_{\sigma_k})\to D^\alpha u$ almost everywhere and in $L^1$. The desired estimate follows immediately.
\end{proof}

We note that the $L^p$ estimate follows quite similarly, and the general case in which lower-order derivatives are involved in the differential operator is treated by iteratively using the $L^p$ estimate and the anisotropic Sobolev-type theorem \ref{Sobolev}, which is quite similar to the standard arguments given in \cite{Camp}. The most general result is then stated as follows:
\begin{corollary}
Suppose that $u\in W^{\mathcal{A},1}(B^\kappa_R)$ solves
\begin{equation}\label{localeq}
(P+H)u
=\sum_{\alpha\in\mathcal{A},\beta\in\mathcal{B}}a_{\alpha\beta}D^{\alpha+\beta}u
+\sum_{\substack{\alpha,\beta\in\mathbb{N}_0^n: \\ \kappa\cdot(\alpha+\beta)<m}}a_{\alpha\beta}D^{\alpha+\beta}u
=\sum_{\beta\in\mathcal{B}}D^\beta f_\beta,
\end{equation}
in the sense of distribution, where the symbol $p(\xi)$ of $P$ satisfies the general requirement (\ref{Varellip}). Then in fact $u\in C^{\mathcal{A}}(B^\kappa_R)$, and the similar estimate as in theorem \ref{VariablePrin} holds.
\end{corollary}

Finally, let us sketch how to deal with systems. We consider the distributions $u$ and $\{f_\beta\}_{\beta\in\mathcal{B}}$ as $\mathbb{C}^N$-valued, and the coefficients $\{a_{\alpha\beta}\}$ as $M_N(\mathbb{C})$-valued functions. The ellipticity condition (\ref{Varellip}) is thus replaced by the following condition:

\emph{As $\xi$ exhausts $\partial B^\kappa_1(0)$ and $x$ exhausts $B^\kappa_{R_0}$, the collection of matrices
$$\sum_{\alpha\in\mathcal{A},\beta\in\mathcal{B}}a_{\alpha\beta}(x)(i\xi)^{\alpha+\beta}$$
remains within a compact subset of $GL_N(\mathbb{C})$.}

For any fixed $x_0\in B^\kappa_R$, the parametrix then takes the form
$$\left[(1-\phi(\xi))\left(\sum_{\alpha\in\mathcal{A},\beta\in\mathcal{B}}a_{\alpha\beta}(x_0)(i\xi)^{\alpha+\beta}\right)^{-1}
\sum_{\beta\in\mathcal{B}}(i\xi)^\beta\hat{f}_\beta(\xi)\right]^\vee.$$
With this parametrix construction, the previous arguments for a single equation remain valid. This enables us to deal with, for example, parabolic systems of any order.

\end{spacing}
\end{document}